\documentclass[final]{siamltex}

\usepackage{epsfig}
\usepackage{amsmath,amssymb}
\usepackage {graphicx}
\usepackage{amsmath,amssymb}

\newtheorem{cor}{Corollary}[section]
\newtheorem{example}{Example}
\newtheorem{rem}{Remark}[section]

\title{On the convergence of harmonic Ritz vectors and harmonic Ritz values}

\author{Gang Wu\thanks{Department of Mathematics,
China University of Mining and Technology, Xuzhou, 221116, Jiangsu, P.R. China.
E-mail: {\tt gangwu76@126.com} and {\tt gangwu@cumt.edu.cn}. This author is
supported by the National Science Foundation of China under grant 11371176, the Natural Science Foundation of Jiangsu
Province under grant BK20131126, the 333 Project of Jiangsu Province, and the Talent Introduction Program of China
University of Mining and Technology.}
}

\begin{document}

\maketitle


\begin{abstract}
We are interested in computing a simple eigenpair $(\lambda,{\bf x})$ of a large non-Hermitian matrix $A$, by a general harmonic Rayleigh-Ritz projection method.
Given a search subspace $\mathcal{K}$ and a target point $\tau$, we focus on the convergence of the harmonic Ritz vector $\widetilde{\bf x}$ and harmonic Ritz value $\widetilde{\lambda}$.
In [{Z. Jia}, {\em The convergence of harmonic Ritz values, harmonic Ritz vectors,
and refined harmonic Ritz vectors}, Math. Comput., 74 (2004), pp. 1441--1456.], Jia showed that for the convergence of harmonic Ritz vector and harmonic Ritz value, it is essential to assume certain Rayleigh quotient matrix being {\it uniformly nonsingular} as $\angle({\bf x},\mathcal{K})\rightarrow 0$.
However, this assumption can not be guaranteed theoretically for a general matrix $A$, and the Rayleigh quotient matrix can be singular or near singular even if $\tau$ is not close to $\lambda$.
In this paper, we abolish this constraint and derive new bounds for the convergence of harmonic Rayleigh-Ritz projection methods. We show that as the distance between ${\bf x}$ and $\mathcal{K}$ tends to zero and $\tau$ is satisfied with the so-called {\it uniform separation condition}, the harmonic Ritz value converges, and the harmonic Ritz vector converges as $\frac{1}{\lambda-\tau}$ is well separated from other Ritz values of $(A-\tau I)^{-1}$ in the orthogonal complement of $(A-\tau I)\widetilde{\bf x}$
with respect to $(A-\tau I)\mathcal{K}$.
\end{abstract}



\pagestyle{myheadings} \thispagestyle{plain} \markboth{GANG WU}{Convergence of harmonic Ritz vectors and harmonic Ritz values}

\section{Introduction}
\setcounter{equation}{0}

Rayleigh-Ritz projection methods are popular techniques for
solving a few eigenpairs of an $n$-by-$n$ large non-Hermitian matrix $A$ \cite{Bai,Saad,Stewart}. Given a search subspace $\mathcal{K}$ whose dimension $m\ll n$, the
Rayleigh-Ritz procedure seeks an approximation $(\widehat{\lambda},\widehat{\bf x})$ to the desired eigenpair $(\lambda,{\bf x})$, which satisfies the following Galerkin condition
\begin{equation}\label{eqn1.1}
 \bigg\{\begin{array}{c}
 \widehat{\bf x}\in \mathcal{K}\\
A\widehat{\bf x}-\widehat{\lambda}\widehat{\bf x}~\bot~\mathcal{K}.
 \end{array}
\end{equation}
Let $V_m$ be an orthonormal basis of $\mathcal{K}$, then \eqref{eqn1.1} reduces to solving a {\it standard} eigenproblem with respect to an $m$-by-$m$ matrix:
$V_m^{\rm H}AV_m{\bf p}=\widehat{\lambda}{\bf p}$.
Then $\widehat{\lambda}$ is called a {Ritz value} and $\widehat{\bf x}=V_m{\bf p}$ is called a {Ritz vector} of $A$ in the subspace $\mathcal{K}$.

It is well known that Rayleigh-Ritz projection methods are suitable for solving exterior eigenpairs \cite{Mor1,Paige}, however,
if some interior eigenpairs are desired, the harmonic Rayleigh-Ritz projection
methods are more appropriate \cite{Bai,Mor1,Mor2,Paige,Stewart}.
Suppose that we want to compute some eigenvalues near a target point $\tau\in\mathbb{C}$. Then these possible interior eigenvalues
of $A$ can be transformed into some exterior eigenvalues of the matrix $(A-\tau I)^{-1}$, and some standard Rayleigh-Ritz projection methods
working on $(A-\tau I)^{-1}$ can be utilized. However, factoring $A-\tau I$
is often very expensive and even impractical. The harmonic Rayleigh-Ritz projection method is primarily used to settle
this problem and to compute interior eigenvalues and eigenvectors of large matrices
without factoring $A-\tau I$ \cite{Fre,Mor1,Mor2,Paige}.

More precisely, the harmonic Rayleigh-Ritz procedure seeks a harmonic Ritz pair $(\widetilde{\lambda},\widetilde{\bf x})$ that satisfies the following Petrov-Galerkin condition \cite{Bai,Saad,Stewart}
\begin{equation}\label{eqn1.2}
 \bigg\{\begin{array}{c}
 \widetilde{\bf x}\in \mathcal{K}\\
A\widetilde{\bf x}-\widetilde{\lambda}\widetilde{\bf x}~\bot~(A-\tau I)\mathcal{K}.
 \end{array}
\end{equation}
Let $V_m$ be an orthonormal basis of $\mathcal{K}$, then \eqref{eqn1.2} reduces to an $m$-by-$m$ {\it generalized} eigenproblem
\begin{equation}\label{eqn1.3}
V_m^{\rm H}(A-\tau I)^{\rm H}(A-\tau I)V_m{\bf q}=(\widetilde{\lambda}-\tau)V_{m}^{\rm H}(A-\tau I)^{\rm H}V_m{\bf q},
\end{equation}
where $\widetilde{\lambda}$ is called a {harmonic Ritz value} and $\widetilde{\bf x}=V_m{\bf q}$ is called a {harmonic Ritz vector} of $A$ in the subspace $\mathcal{K}$.
Notice that $V_m^{\rm H}(A-\tau I)^{\rm H}(A-\tau I)V_m$ is Hermitian positive definite provided $\tau$ is not an eigenvalue of $A$, and thus the matrix pencil in \eqref{eqn1.3} is always regular \cite{GV}.
We remark that the harmonic projection is well defined and can
compute $m$ harmonic Ritz pairs only if the matrix pencil in \eqref{eqn1.3} is regular. Moreover, in this case, there might be infinite harmonic Ritz values.

In this paper, we are concerned with a simple eigenpair $(\lambda,{\bf x})$, and focus on the convergence of the harmonic Ritz pair $(\widetilde{\lambda},\widetilde{\bf x})$ as the distance between ${\bf x}$ and $\mathcal{K}$ approaches zero.
Denote
$$
B=V_m^{\rm H}(A-\tau I)^{\rm H}V_m,\quad C=V_m^{\rm H}(A-\tau I)^{\rm H}(A-\tau I)V_m,
$$
if $B$ is nonsingular, then \eqref{eqn1.3} can be rewritten as
$(B^{-1}C){\bf q}=(\widetilde{\lambda}-\tau){\bf q}$.
The matrix $B^{\rm H}$ is called the {\it Rayleigh quotient matrix} and $B^{-1}C$ the {\it harmonic
Rayleigh quotient matrix} of $A$ with respect to $\mathcal{K}$ and the target point $\tau$, respectively.
In \cite{Jia}, it was shown that the harmonic Ritz value $\widetilde{\lambda}$ converges to $\lambda$ if a certain Rayleigh
quotient matrix is {\it uniformly nonsingular}; and the harmonic Ritz vector $\widetilde{\bf x}$
converges to ${\bf x}$ if the Rayleigh quotient matrix is {\it uniformly nonsingular} and $\widetilde{\lambda}$
remains well separated from the other harmonic Ritz values; see also \cite{Chen}.
However, for a non-Hermitian matrix $A$, the assumption of the Rayleigh quotient matrix is {uniformly nonsingular} can not be guaranteed theoretically
\cite[Theorem 4.3]{Jia}.
Further, the Rayleigh quotient matrix can be singular or near singular even if $\tau$ is not close to $\lambda$, and the upper bounds given in \cite{Jia} can be overestimated when the Rayleigh quotient matrix is near singular. We will give an example to illustrate this phenomenon. Therefore, it is necessary to revisit the convergence of harmonic Ritz vectors and harmonic Ritz values.

Very recently, Vecharynski \cite{Vec} presented a Saad's type bound for harmonic Ritz vectors of
a Hermitian matrix $A$. It shows that, along with $\angle({\bf x},\mathcal{K})$, the closeness of the
harmonic Ritz vectors to the exact eigenvectors generally depends on the spectral condition number of $A-\tau I$.
In this paper, we abolish the constraint of uniform non-singularity of the Rayleigh
quotient matrix that required in \cite{Chen,Jia}, and generalize the main result of \cite{Vec} to the non-Hermitian case.
We prove that as $\angle({\bf x},\mathcal{K})\rightarrow 0$, the closeness of the
harmonic Ritz pair $(\widetilde{\lambda},\widetilde{\bf x})$ to the exact eigenpair $(\lambda,{\bf x})$, generally depends on the condition number of the shifted matrix $A-\tau I$, which is uniformly bounded as $\tau\neq\lambda$.
More precisely, we show that the convergence of the harmonic Ritz vector is closely related to the condition number $\kappa(A-\tau I)$, while the convergence of the harmonic Ritz value generally depends on the ratio ${\kappa(A-\tau I)}/{|\tau-\lambda|}$.
Note that the assumption of $\angle({\bf x},\mathcal{K})\rightarrow 0$ is not stringent in practice. Indeed, it is a necessary condition for {\it any} approximate eigenvectors in $\mathcal{K}$ to converge. Otherwise, {all} the approximate eigenvectors in $\mathcal{K}$ will not converge \cite{Jia,JS,JS2}.

Throughout this paper, we make the assumption that $\lambda\neq \tau$ and $A-\tau I$ is nonsingular, and $\widetilde{\lambda}$ is not an infinite eigenvalue of the matrix pencil $(C,B)$.
Moreover, we suppose that $\lambda$ is the {\it closet} eigenvalue to $\tau$, i.e.,
$$
|\lambda-\tau|=\min_{\lambda_i\in\Lambda(A)}|\lambda_i-\tau|,
$$
where $\Lambda(A)=\{\lambda_1,\lambda_2,\ldots,\lambda_n\}$ is the spectrum of $A$.
The eigenvector and the approximate eigenvectors are normalized to have
unit length. The norm $\|\cdot\|$ used is the Euclidean norm (or the 2-norm) of a vector or matrix, and $\angle({\bf x},{\bf y})$ represents the {acute} angle between the two vectors ${\bf x}$ and ${\bf y}$. The angle between a vector ${\bf x}$ and a subspace $\mathcal{K}$ is defined as $\angle({\bf x},\mathcal{K})=\min_{{\bf w}\in\mathcal{K},{\bf w}\neq {\bf 0}}\angle({\bf x},{\bf w})$, and the sine of the angle between a vector ${\bf x}$ and a subspace $\mathcal{K}$ is defined as \cite[pp.250]{Stewart}
$$
\sin\angle({\bf x},\mathcal{K})=\min_{{\bf y}\in\mathcal{K}}\|{\bf x}-{\bf y}\|.
$$
Hence, for an unit vector ${\bf x}$ and an arbitrary ${\bf y}\in \mathbb{C}^{n}$, we have
\begin{equation}\label{1.1}
\sin\angle({\bf x},{\bf y})=\min_{\alpha\in\mathbb{C}}\|{\bf x}-\alpha{\bf y}\|.
\end{equation}
Denote by the superscript ``H" the conjugate
transpose of a matrix or vector.
Let $\kappa(A)=\sigma_{\max}(A)/\sigma_{\min}(A)$ the condition number of $A$, where
$\sigma_{\min}(A)$ and $\sigma_{\max}(A)$ stand for the smallest and the largest singular value of $A$, respectively.

\section{Some related work}
\setcounter{equation}{0}
There are a lot of convergence results available on standard Rayleigh-Ritz projection methods; see \cite{JS,JS2,JS3,Saad,Stewart} and the references therein.
In \cite{Stewart1}, a general {\it a prior} bound that describes the approximation quality of Ritz vectors for non-Hermitian matrices was given by Stewart, which is a generalization of ``Saad's bound" \cite[Theorem 4.6]{Saad} for Hermitian matrices. We call it ``Stewart's bound" on Ritz vectors of general matrices.
\begin{theorem}{\rm(}Stewart's bound \cite{Stewart}{\rm)}\label{Thm1}
Assume that $(\widehat{\lambda},\widehat{\bf x})$ is a Ritz pair of the matrix $A$ to the subspace $\mathcal{K}$ and that the matrix $U$ having orthonormal columns span the orthogonal complement of $\widehat{\bf x}$ with respect to $\mathcal{K}$. Define $G=U^{\rm H}AU$ and ${\rm sep}(\lambda,G)=\min_{\|u\|=1}\|Gu-\lambda{\bf u}\|$. Then
\begin{equation}
\sin\angle({\bf x},\widehat{\bf x})\leq\sin\angle({\bf x},\mathcal{K})\sqrt{1+\frac{\gamma^2}{{\rm sep}(\lambda,G)^2}},
\end{equation}
where $\gamma=\|P_{\mathcal{K}}A(I-P_{\mathcal{K}})\|$ and $P_{\mathcal{K}}$ is the orthogonal projector on $\mathcal{K}$.
\end{theorem}

Next we present some related work on the convergence of harmonic Ritz vectors and harmonic Ritz values.
Let $[{\bf q},~Q_{\bot}]$ be unitary, then
$$
\left[\begin{array}{ccc}
{\bf q}^{\rm H}\\
Q_{\bot}^{\rm H}\\
\end{array} \right]B^{-1}C\big[{\bf q},~Q_{\bot}]=\left[\begin{array}{ccc}
\widetilde{\lambda}-\tau& {\bf g}^{\rm H}\\
{\bf 0}&G_1\\
\end{array} \right],
$$
where ${\bf g}^{\rm H}={\bf q}^{\rm H}(B^{-1}C)Q_{\bot}$ and $G_1=Q_{\bot}^{\rm H}(B^{-1}C)Q_{\bot}$. Assume that $B$ is nonsingular and let $\tau=0$, Chen and Jia \cite{Chen} extended the
result of Theorem \ref{Thm1} to the harmonic Ritz vectors, and established the following {\it a priori} error bound
for $\sin\angle({\bf x},\widetilde{\bf x})$ in terms of $\sin\angle({\bf x},\mathcal{K})$.
\begin{theorem}\label{Thm1.2}
Assume that $B$ is nonsingular and let $\tau=0$. Let $(\widetilde{\lambda},\widetilde{\bf x})$ be a harmonic Ritz pair with respect to the subspace $\mathcal{K}$. Then
\begin{equation}
\sin\angle({\bf x},\widetilde{\bf x})\leq\sin\angle({\bf x},\mathcal{K})\sqrt{1+\frac{\gamma_1^2\|B^{-1}\|^2}{{\rm sep}(\lambda,G_1)^2}},
\end{equation}
where $\gamma_1=\|P_{\mathcal{K}}A^{\rm H}(\lambda I-A)(I-P_{\mathcal{K}})\|$.
\end{theorem}

In \cite{Jia}, Jia gave insight into the convergence of harmonic Ritz vectors and harmonic Ritz values for the harmonic Rayleigh-Ritz projection methods in a unified manner, and pointed out that both $\widetilde{\lambda}$ and $\widetilde{\bf x}$ converge to $\lambda$ and ${\bf x}$ only {\it conditionally} as $\sin\angle({\bf x},\mathcal{K})\rightarrow 0$.
The following theorem gives the convergence of harmonic Ritz vectors.
\begin{theorem}\cite[Theorem 3.2]{Jia}\label{Thm1.3}
Let $\epsilon=\sin\angle({\bf x},\mathcal{K})$, $\tau=0$, and assume that $B$ is nonsingular. If ${\rm sep}(\lambda,G_1)>0$, then
\begin{equation}\label{23}
\sin\angle({\bf x},\widetilde{\bf x})\leq\left(1+\frac{2\|B^{-1}\|\|A\|^2}{\sqrt{1-\epsilon^2}~{\rm sep}(\lambda,G_1)}\right)\epsilon.
\end{equation}
\end{theorem}

In particular, we have the following corollary. It indicates that the convergence of harmonic Ritz vectors is closely related to the condition number $\kappa(A)$ {\big(}or $\kappa(A-\tau I)$ if $\tau\neq 0${\big)} when $A$ {\big(}or $A-\tau I${\big)} is Hermitian positive or Hermitian negative \cite{Jia}.
\begin{cor}
Let $\epsilon=\sin\angle({\bf x},\mathcal{K})$ and $\tau=0$. Assume that $A$ is Hermitian positive or Hermitian negative. If ${\rm sep}(\lambda,G_1)>0$, then
$$
\sin\angle({\bf x},\widetilde{\bf x})\leq\left(1+\frac{2\kappa(A)\|A\|}{\sqrt{1-\epsilon^2}~{\rm sep}(\lambda,G_1)}\right)\epsilon.
$$
\end{cor}

A combination of the following two theorems sheds light on the convergence of harmonic Ritz values.
\begin{theorem}\cite[Theorem 2.1]{Jia}\label{Thm1.4}
Let $B^{-1}C$ be the harmonic Rayleigh quotient matrix, let $(\lambda,{\bf x})$ be the desired eigenpair of $A$ with $\|{\bf x}\|=1$ and let $\epsilon=\sin\angle({\bf x},\mathcal{K})$, $\tau=0$. Then there is a matrix $E$ satisfying
\begin{equation}\label{eqn17}
\|E\|\leq\frac{\epsilon}{\sqrt{1-\epsilon^2}}\|B^{-1}\|(|\lambda|\|A\|+\|A\|^2)
\end{equation}
and such that $\lambda$ is an eigenvaule of $B^{-1}C+E$.
\end{theorem}

\begin{theorem}\cite[Corollary 2.2]{Jia}\label{Thm1.5}
There is an eigenvalue $\mu$ of $B^{-1}C$ such that
\begin{equation}
|\lambda-\mu|\leq(2\|A\|+\|E\|)^{1-\frac{1}{m}}\|E\|^{\frac{1}{m}},
\end{equation}
where $m$ is the order of $B^{-1}C$.
\end{theorem}

All the results given in Theorems \ref{Thm1.2}--\ref{Thm1.5} require that $\|B^{-1}\|$ is {\it uniformly bounded} (or $B$ is {\it uniformly nonsingular}) as $\sin\angle({\bf x},\mathcal{K})\rightarrow 0$, which can not be guaranteed theoretically \cite[Theorem 4.3]{Jia}. For instance, for interior eigenvalues of an indefinite (shifted) $A$, this condition might never be satisfied \cite[pp.104]{Voe}.
In practice, the matrix $B$ can be singular or near singular. Even if $B$ is not singular exactly, near singularity is bad enough \cite{Jia}.
Recently, Vecharynski established a Saad's type bound for harmonic Ritz vectors of a Hermitian matrix \cite{Vec}.
The bound shows a dependence of the harmonic Rayleigh-Ritz procedure
on the spectral condition number of $A-\tau I$. Let's discuss it in more detail.
First, a two-sided bound on $\sin\angle({\bf x},A{\bf y})$ in terms of $\sin\angle({\bf x},{\bf y})$ was investigated when $A$ is Hermitian, where ${\bf y}\in\mathbb{C}^n$ is an arbitrary nonzero vector. It is crucial for deriving the main result of \cite[pp.4]{Vec}.
\begin{theorem}\cite[Lemma 1]{Vec}\label{Thm2.2}
Let $(\lambda,{\bf x})$ be an eigenpair of a nonsingular Hermitian matrix $A$. Then for any vector ${\bf y}$, we have
\begin{equation}\label{eqn2.10}
\left|\frac{\lambda_{\min}}{\lambda}\right|\sin\angle({\bf x},{\bf y})\leq\sin\angle{({\bf x},A{\bf y})}\leq\left|\frac{\lambda_{\max}}{\lambda}\right|\sin\angle({\bf x},{\bf y}),
\end{equation}
where $\lambda_{\min}$ and $\lambda_{\max}$ are the smallest and largest magnitude eigenvalues of $A$.
\end{theorem}

Second, the main result of \cite{Vec} was established. It shows that the approximation quality of the harmonic Rayleigh-Ritz procedure can be hindered by a poor conditioning of $A-\tau I$. For instance,
this can happen if the shift $\tau$ is chosen to be very close to an eigenvalue of $A$.
\begin{theorem}\cite[Theorem 3]{Vec}\label{Thm2.7}
Let $(\lambda,{\bf x})$ be an eigenpair of a Hermitian matrix $A$ and $(\widetilde{\lambda},\widetilde{\bf x})$ be a
harmonic Ritz pair with respect to the subspace $\mathcal{K}$ and shift $\tau\notin\Lambda(A)=\{\lambda_1,\lambda_2,\ldots,\lambda_n\}$. Assume
that $\Theta$ is a set of all the harmonic Ritz values and let $P_\mathcal{Q}$ be an orthogonal
projector onto $(A-\tau I)\mathcal{K}$. Then
\end{theorem}
$$
\sin\angle({\bf x},\widetilde{\bf x})\leq\kappa(A-\tau I)\sqrt{1+\frac{\gamma^2}{\delta^2}}\sin\angle({\bf x},\mathcal{K}),
$$
where $\gamma=\|P_{\mathcal{Q}}(A-\tau I)^{-1}(I-P_\mathcal{Q})\|$,
$$
\kappa(A-\tau I)=\frac{\max_{\lambda_j\in\Lambda(A)}|\lambda_j-\tau|}{\min_{\lambda_j\in\Lambda(A)}|\lambda_j-\tau|},
$$
and
$$
\delta=\min_{\widetilde{\lambda}_j\in\Theta\backslash\lambda}\left|\frac{1}{\lambda-\tau}-\frac{1}{\widetilde{\lambda}_j-\tau}\right|.
$$

In this work, we abolish the constraint of uniform non-singularity of $B$ that required in Theorem \ref{Thm1.2}--Theorem \ref{Thm1.5}, and generalize Theorem \ref{Thm2.7} to non-Hermitian matrices. New bounds for the convergence of harmonic Ritz vectors and harmonic Ritz values are derived, which shows that the convergence of harmonic Ritz vectors is closely related to the condition number $\kappa(A-\tau I)$, while the convergence of harmonic Ritz values generally depends on the ratio ${\kappa(A-\tau I)}/{|\tau-\lambda|}$. The convergence of refined harmonic Ritz vector is also discussed.

\section{An extension of Stewart's bound on harmonic Ritz vectors}
\setcounter{equation}{0}

In this section, we take into account the convergence of harmonic Ritz vector as the distance between ${\bf x}$ and $\mathcal{K}$ approaches zero.
We first present the following theorem that is a generalization of Theorem \ref{Thm2.2} to non-Hermitian matrices.
Note that the proof is much simpler than that of Theorem \ref{Thm2.2}; see \cite[pp.4--pp.6]{Vec}.
\begin{theorem}\label{Lem3.2}
Let $(\lambda,{\bf x})$ be an eigenpair of $A$, and suppose that $\tau\neq\lambda$. Then for any nonzero vector ${\bf y}\in\mathbb{C}^{n}$, we have
\begin{equation}\label{3222}
\frac{\sigma_{\min}(A-\tau I)}{|\lambda-\tau|}\sin\angle({\bf x},{\bf y})\leq\sin\angle\big({\bf x},(A-\tau I){\bf y}\big)\leq\frac{\sigma_{\max}(A-\tau I)}{|\lambda-\tau|}\sin\angle({\bf x},{\bf y}).
\end{equation}
\end{theorem}
\begin{proof}
On one hand, we have from \eqref{1.1} that
\begin{eqnarray}\label{eqn322}
\sin\angle\big({\bf x},(A-\tau I){\bf y}\big)&=&\sin\angle\big((A-\tau I){\bf x}/(\lambda-\tau),(A-\tau I){\bf y}\big)\nonumber\\
&=&\min_{\alpha\in\mathbb{C}}\frac{1}{|\lambda-\tau|}\|(A-\tau I){\bf x}-\alpha (A-\tau I){\bf y}\|\nonumber\\
&=&\frac{1}{|\lambda-\tau|}\min_{\alpha\in\mathbb{C}}\|(A-\tau I)({\bf x}-\alpha {\bf y})\|\nonumber\\
&\leq&\frac{\|A-\tau I\|}{|\lambda-\tau|}\min_{\alpha\in\mathbb{C}}\|{\bf x}-\alpha {\bf y}\|\nonumber\\
&=&\frac{\sigma_{\max}(A-\tau I)}{|\lambda-\tau|}\sin\angle\big({\bf x},{\bf y}\big).
\end{eqnarray}
On the other hand, we have
\begin{eqnarray*}
\sin\angle({\bf x},{\bf y})&=&\min_{\alpha\in\mathbb{C}}\|{\bf x}-\alpha {\bf y}\|\nonumber\\
&=&\min_{\alpha\in\mathcal{C}}\|(A-\tau I)^{-1}(A-\tau I)({\bf x}-\alpha {\bf y})\|\nonumber\\
&\leq&\|(A-\tau I)^{-1}\|\min_{\alpha\in\mathbb{C}}\|(A-\tau I){\bf x}-\alpha (A-\tau I){\bf y}\|\nonumber\\
&=&\|(A-\tau I)^{-1}\|\min_{\alpha\in\mathbb{C}}\|(\lambda-\tau){\bf x}-\alpha (A-\tau I){\bf y}\|\nonumber\\
&=&|\lambda-\tau|\|(A-\tau I)^{-1}\|\cdot\min_{\alpha\in\mathbb{C}}\|{\bf x}-\alpha (A-\tau I){\bf y}\|\nonumber\\
&=&|\lambda-\tau|\|(A-\tau I)^{-1}\|\cdot\sin\angle\big({\bf x}, (A-\tau I){\bf y}\big),
\end{eqnarray*}
and thus
\begin{equation}\label{eqn333}
\sin\angle\big({\bf x}, (A-\tau I){\bf y}\big)\geq\frac{\sigma_{\min}(A-\tau I)}{|\lambda-\tau|}\sin\angle({\bf x},{\bf y}).
\end{equation}
A combination of \eqref{eqn322} and \eqref{eqn333} yields \eqref{3222}.
\end{proof}

We are ready to consider the convergence of harmonic Ritz vector. If $\widetilde{\lambda}\neq\tau$, then \eqref{eqn1.3} can be rewritten as
\begin{equation}\label{eqn2.16}
V_{m}^{\rm H}(A-\tau I)^{\rm H}(A-\tau I)^{-1}(A-\tau I)V_m{\bf q}=\frac{1}{\widetilde{\lambda}-\tau}V_m^{\rm H}(A-\tau I)^{\rm H}(A-\tau I)V_m{\bf q},
\end{equation}
or equivalently,
\begin{equation}\label{217}
\big[(A-\tau I)V_{m}\big]^{\rm H}\left[(A-\tau I)^{-1}\Big((A-\tau I)V_m{\bf q}\Big)-\frac{1}{\widetilde{\lambda}-\tau}\Big((A-\tau I)V_m{\bf q}\Big)\right]={\bf 0}.
\end{equation}
In other words, \eqref{eqn2.16} can be understood as an orthogonal projection process for $(A-\tau I)^{-1}$ with respect to the subspace $(A-\tau I)\mathcal{K}$, and the Ritz pair in $(A-\tau I)\mathcal{K}$ is given by $\big(1/(\widetilde{\lambda}-\tau),(A-\tau I)V_m{\bf q}\big)$. Recall that the harmonic Ritz value and harmonic Ritz vector of $A$ in the subspace $\mathcal{K}$ are given by $\widetilde{\lambda}$ and $\widetilde{\bf x}=V_m{\bf q}$, respectively.
The following theorem can be viewed as an extension of Stewart's bound on harmonic Ritz vectors, or a generalization of Theorem \ref{Thm2.7} to non-Hermitian matrices.
\begin{theorem}\label{Thm21}
Let $\widehat{U}$
be an orthonormal matrix whose columns span the orthogonal complement of $(A-\tau I)\widetilde{\bf x}$
with respect to $(A-\tau I)\mathcal{K}$, and let $\widehat{G}=\widehat{U}^{\rm H}(A-\tau I)^{-1}\widehat{U}$. Then we have
\begin{equation}\label{eqn2.20}
\sin\angle({\bf x},\widetilde{\bf x})\leq\kappa(A-\tau I)\sqrt{1+\frac{\widehat{\gamma}^2}{{\rm sep}\big(1/(\lambda-\tau),\widehat{G}\big)^2}}\cdot\sin\angle({\bf x},\mathcal{K}),
\end{equation}
where $\widehat\gamma=\|P_{\mathcal{Q}}(A-\tau I)^{-1}(I-P_{\mathcal{Q}})\|$ with $P_{\mathcal{Q}}$ being the orthogonal projector on the subspace $(A-\tau I)\mathcal{K}$.
\end{theorem}
\begin{proof}
As $\big(1/(\lambda-\tau),{\bf x}\big)$ is an eigenpair of $(A-\tau I)^{-1}$, if ${\rm sep}\big(1/(\lambda-\tau),\widehat{G}\big)>0$, it follows from Theorem \ref{Thm1} that
\begin{equation}\label{eqn2.17}
\sin\angle\big({\bf x},(A-\tau I)\widetilde{\bf x}\big)\leq\sin\angle\big({\bf x},(A-\tau I)\mathcal{K}\big)\sqrt{1+\frac{\widehat{\gamma}^2}{{\rm sep}\big(1/(\lambda-\tau),\widehat{G}\big)^2}}.
\end{equation}
Next we relate $\sin\angle({\bf x},(A-\tau I)\mathcal{K})$ to $\sin\angle({\bf x},\mathcal{K})$ for a general matrix $A$.
Let $\widehat{\bf y}\in\mathcal{K}$ such that $\angle({\bf x},\widehat{\bf y})=\angle({\bf x},\mathcal{K})$.
Since $A-\tau I$ is nonsingular and $\angle({\bf x},\widehat{\bf y})=\angle({\bf x},\mathcal{K})$, we have from Theorem \ref{Lem3.2} that
\begin{eqnarray}\label{36}
\sin\angle\big({\bf x},(A-\tau I)\mathcal{K}\big)&\leq&\sin\angle{\big({\bf x},(A-\tau I)\widehat{\bf y}\big)}\nonumber\\
&\leq&\frac{\sigma_{\max}(A-\tau I)}{|\lambda-\tau|}\sin\angle({\bf x},\widehat{\bf y})\nonumber\\
&=&\frac{\sigma_{\max}(A-\tau I)}{|\lambda-\tau|}\sin\angle({\bf x},\mathcal{K}).
\end{eqnarray}
Furthermore, we obtain from Theorem \ref{Lem3.2} that
\begin{equation}\label{37}
\sin\angle({\bf x},(A-\tau I)\widetilde{\bf x})\geq\frac{\sigma_{\min}(A-\tau I)}{|\lambda-\tau|}\sin\angle({\bf x},\widetilde{\bf x}).
\end{equation}
Combining \eqref{eqn2.17}--\eqref{37}, we get \eqref{eqn2.20} on the convergence of harmonic Ritz vector.
\end{proof}

\begin{rem}
Under the condition that $1/(\lambda-\tau)$ is well separated from the eigenvalues of $\widehat{G}$, Theorem \ref{Thm21} shows that the harmonic Ritz vector converges as $\angle({\bf x},\mathcal{K})\rightarrow 0$ and $\tau$ is not very close to $\lambda$.
By ``$\tau$ is not very close to $\lambda$", we mean
\begin{equation}
\kappa(A-\tau I)\sin\angle({\bf x},\mathcal{K})\ll 1,
\end{equation}
which is equivalent to
\begin{equation}\label{14}
\sigma_{\min}(A-\tau I)\gg\|A-\tau I\|\sin\angle({\bf x},\mathcal{K}).
\end{equation}
We refer to \eqref{14} as the ``{\it uniform separation condition}" of $\tau$ with respect to $\lambda$. As $|\lambda-\tau|\geq\sigma_{\min}(A-\tau I)$,
\eqref{14} means $|\lambda-\tau|$ is not small when $A$ is normal, and vice versa. In the non-normal case, however, \eqref{14} implies $|\lambda-\tau|$ is not small, but not vice versa.

Similar sufficient conditions are required to guarantee the convergence of harmonic Ritz vector \cite{Chen,Jia}.
However, under the condition that $\angle({\bf x},\mathcal{K})\rightarrow 0$ and $\lambda$ is well separated from the eigenvalues of $G_1$, both Theorem \ref{Thm1.2} and Theorem \ref{Thm1.3} require the condition that
$\|B^{-1}\|$ is uniformly bounded, which might not be satisfied. Further, the matrix $B$ can be singular or near singular even if $\tau$ is far away from $\lambda$, as the following example illustrates.



\end{rem}

\begin{example}
Consider the matrix
$$
A=\left[\begin{array}{ccc}
2 & 2 & 3\\ 0 & 1 & 4\\ 0 & 6 & 6
\\\end{array} \right],
$$
whose eigenvalues are $\{2,-2,9\}$. We want to compute the eigenvalue $\lambda=2$ and the associated eigenvector ${\bf x}=[1,0,0]^{\rm H}$. If we choose $\tau=1$ and let
$$
V_m=\left[\begin{array}{cc}
1 & 0\\ 0& 1\\ 0& 0
\end{array} \right],
$$
then $\sin\angle({\bf x},\mathcal{K})=0$ with $\mathcal{K}={\rm span}\{V_m\}$. Notice that $|\lambda-\tau|=1\gg 0$, moreover,
$$
B=V_m^{\rm H}(A-\tau I)^{\rm H}V_m=\left[\begin{array}{cc}
1 & 0\\ 2& 0
\end{array}\right]
$$
which is singular, and
$$
C=V_m^{\rm H}(A-\tau I)^{\rm H}(A-\tau I)V_m=\left[\begin{array}{cc}
1 & 2\\ 2& 40
\end{array}\right].
$$
The matrix pencil $(C,B)$ has two eigenvalues: 1 and -{\rm Inf}, i.e., the matrix pencil has an infinite eigenvalue.
All the results given in Theorem \ref{Thm1.2}--Theorem \ref{Thm1.5} do not work.
However, we find that $\kappa(A-\tau I)=10.006189420283654$, implying that $A-\tau I$ is perfectly conditioned. Indeed, the harmonic Ritz vector is $\widetilde{\bf x}=[1,0,0]^{\rm H}$ and the harmonic Ritz value is $\widetilde{\lambda}=2$, both are exact solutions.

Of course, in practice it is rare that ${\bf x}\in\mathcal{K}$ and $B$ is singular exactly. But a nearly singular $B$ is bad enough. We next perturb $V_m$ by setting the $(3,1)$ element of $V_m$ to $\varepsilon=10^{-6}$. Then we have the orthonormalized
$$
V_m=\left[\begin{array}{cc}
 0.999999999999500      &             0\\
                   0  & 1.000000000000000\\
   0.000001000000000     &              0
\end{array} \right].
$$
In this case, we have $\sin\angle({\bf x},\mathcal{K})=9.999999999995000\times 10^{-7}$,
$$
B=\left[\begin{array}{cc}
 1.000003000004000 &  0.000004000000000\\
   2.000005999999000   &              0
\end{array}\right],
$$
$$
C=\left[\begin{array}{cc}
  1.000006000049000 &  2.000035999999000\\
   2.000035999999000 & 40.000000000000000
\end{array}\right],
$$
and $\|B^{-1}\|= 2.795084971879544\times 10^5$. In terms of Theorem \ref{Thm1.2}--Theorem \ref{Thm1.5}, $\widetilde{\lambda}$ and $\widetilde{\bf x}$ will be poor approximations to $\lambda$ and ${\bf x}$, as the upper bound given by \eqref{23} is $2.830019056701876$. This means that the harmonic Ritz vector will have no accuracy at all.

However, the upper bound given by \eqref{eqn2.20} is $1.084005507313452\times 10^{-5}$, implying that the harmonic Ritz vector $\widetilde{\bf x}$ will be a good approximation to ${\bf x}$.
Indeed, the harmonic Ritz value $\widetilde{\lambda}=2.000001666687963$, the corresponding harmonic Ritz vector
$$
\widetilde{\bf x}=\left[\begin{array}{c}
  0.999999999999500  \\
  -0.000000666665352  \\
   0.000001000000000
\end{array} \right],
$$
and $\sin\angle({\bf x},\widetilde{\bf x})= 1.201849695848198\times 10^{-6}$. Thus, the harmonic Ritz pair $(\widetilde{\lambda},\widetilde{\bf x})$ is an excellent approximation to the desired eigenpair $(\lambda,{\bf x})$.
\end{example}


\section{On the convergence of harmonic Ritz values}
\setcounter{equation}{0}

In \cite{Jia}, Jia analyzed the convergence of harmonic Ritz value, and concluded that it converges as $\angle(x,\mathcal{K})\rightarrow 0$ and $B$ is uniformly nonsingular. However, as Example 1 illustrates, the uniform non-singularity of $B$ cannot be guaranteed theoretically, and $B$ may be singular or near singular even if $\tau$ is not close to $\lambda$ at all.
Thus, it is necessary to establish new theoretical bounds on the convergence of harmonic Ritz values.
In \cite[Theorem 4.1]{Jia}, better error bounds were derived, under the condition that the harmonic Ritz vector $\widetilde{\bf x}$ converges and $|\lambda-\tau|\gg\sin\angle({\bf x},\mathcal{K})$; and it was demonstrated that the harmonic projection method can miss the eigenvalue $\lambda$ if it is very close
to $\tau$.
In this section, we revisit the convergence of the harmonic Ritz value, and show that it converges as $\sin\angle(x,\mathcal{K})\rightarrow 0$ and the uniform separation condition \eqref{14} is satisfied, with no need to make the assumption that $\widetilde{\bf x}\rightarrow {\bf x}$. It is crucial to keep in mind that our new bounds
do not require the evaluation of $(A-\tau I)^{-1}$ and are computable by using $A$ only.

Recall from \eqref{eqn2.16} that  ${1}/(\widetilde{\lambda}-\tau)$ is a Ritz value of $(A-\tau I)^{-1}$ in the subspace $(A-\tau I)\mathcal{K}$.
Let $W_m$ be an orthonormal basis of $(A-\tau I)\mathcal{K}$, and let $D=W_m^{\rm H}(A-\tau I)^{-1}W_m$,
then \eqref{eqn2.16} can be rewritten as the following {\it standard} eigenproblem
\begin{equation}\label{51}
D{\bf q}=\frac{1}{\widetilde{\lambda}-\tau}{\bf q},
\end{equation}
that is, ${1}/(\widetilde{\lambda}-\tau)$ is a Ritz value of $(A-\tau I)^{-1}$ and $W_m{\bf q}$ is the associated Ritz vector in the subspace $(A-\tau I)\mathcal{K}$. Notice that $1/({\widetilde{\lambda}-\tau})$ is uniformly bounded as
$
|1/({\widetilde{\lambda}-\tau})|\leq\|D\|\leq \|(A-\tau I)^{-1}\|.
$
We mention that the reciprocals of the (shifted) harmonic Ritz values
have an application in the computation of eigenvalues in the interior of the
spectrum of a large sparse matrix \cite{Voe,Voe2}.
In \cite{Voe}, V\"{o}emel investigated harmonic Ritz values of a symmetric matrix $A$ from a slightly different perspective, and
considered a bounded functional that yields the reciprocals of the harmonic Ritz values.

By \cite[Theorem 2.1]{JS}, there is a matrix $F$ satisfying
\begin{equation}\label{eqn3.2}
\|F\|\leq\frac{\sin\angle\big({\bf x},(A-\tau I)\mathcal{K}\big)}{\sqrt{1-\sin^2\angle\big({\bf x},(A-\tau I)\mathcal{K}\big)}}\|(A-\tau I)^{-1}\|,
\end{equation}
such that ${1}/({\lambda-\tau})$ is an eigenvalue of $D+F$. Unfortunately, the upper bound established in \eqref{36} is unsuitable to apply to \eqref{eqn3.2} directly, since for a small but fixed $\sin\angle({\bf x},\mathcal{K})$, the righthand side of \eqref{36} can be much larger than 1 as $\sigma_{\max}(A-\tau I)\gg |\lambda-\tau|$. Thus, we need to seek a new upper found for $\sin\angle\big({\bf x},(A-\tau I)\mathcal{K}\big)$ in terms of $\sin\angle({\bf x},\mathcal{K})$.

Let $\widehat{X}\in\mathbb{C}^{n\times (n-1)}$ be an orthonormal basis for the orthogonal complement of the space spanned by ${\bf x}$, such that $[{\bf x},~\widehat{X}]$ is unitary. Then,
\begin{equation}\label{eqn2.1}
(A-\tau I)\big[{\bf x},~\widehat{X}\big]=\big[{\bf x},~\widehat{X}\big]\left[\begin{array}{ccc}
\lambda-\tau& {\bf w}^{\rm H}\\
{\bf 0}&R\\
\end{array} \right],
\end{equation}
where
$
{\bf w}^{\rm H}={\bf x}^{\rm H}A\widehat{X}\in\mathbb{C}^{1\times (n-1)},
$
and
\begin{equation}\label{222}
\|{\bf w}\|=\|{\bf x}^{\rm H}A\widehat{X}\|=\|{\bf x}{\bf x}^{\rm H}A\widehat{X}\|,
\end{equation}
i.e., the norm of ${\bf w}$ equals to that of the orthogonal projection of $A\widehat{X}$ on the invariant subspace ${\rm span}\{{\bf x}\}$.

For any unit vector ${\bf y}\in\mathbb{C}^n$ that is linearly independent on ${\bf x}$, we decompose ${\bf y}=\alpha{\bf x}+\widehat{X}{\bf z}$, where
\begin{equation}\label{eqn2.2}
\sin\angle({\bf x},{\bf y})=\|{\bf z}\|>0,\quad{\rm and}\quad |\alpha|=\cos\angle({\bf x},{\bf y}).
\end{equation}
In terms of \eqref{eqn2.1}, we get
$$
(A-\tau I){\bf y}=\alpha(\lambda-\tau){\bf x}+(A-\tau I)\widehat{X}{\bf z}=\big[\alpha(\lambda-\tau)+{\bf w}^{\rm H}{\bf z}\big]{\bf x}+\widehat{X}(R{\bf z}).
$$
Hence,
\begin{eqnarray}\label{eqn2.3}
\sin\angle\big({\bf x},(A-\tau I){\bf y}\big)&=&\frac{\|(I-{\bf x}{\bf x})^{\rm H}(A-\tau I){\bf y}\|}{\|(A-\tau I){\bf y}\|}
=\frac{\|\widehat{X}\widehat{X}^{\rm H}(A-\tau I){\bf y}\|}{\sqrt{|\alpha(\lambda-\tau)+{\bf w}^{\rm H}{\bf z}|^2+\|R{\bf z}\|^2}}\nonumber\\
&=&\frac{\|R{\bf z}\|}{\sqrt{|\alpha(\lambda-\tau)+{\bf w}^{\rm H}{\bf z}|^2+\|R{\bf z}\|^2}}=\frac{1}{\sqrt{1+\frac{|\alpha(\lambda-\tau)+{\bf w}^{\rm H}{\bf z}|^2}{\|R{\bf z}\|^2}}},
\end{eqnarray}
where we used $\|R{\bf z}\|\neq 0$ as $R$ is nonsingular and ${\bf z}$ is nonzero. In particular, when $A$ is normal, we have ${\bf w}={\bf 0}$ and thus
\begin{equation}
\sin\angle\big({\bf x},(A-\tau I){\bf y}\big)=\frac{1}{\sqrt{1+\frac{|\alpha(\lambda-\tau)|^2}{\|R{\bf z}\|^2}}}.
\end{equation}

On one hand, we have from \eqref{eqn2.2} that
\begin{eqnarray*}
|\alpha(\lambda-\tau)+{\bf w}^{\rm H}{\bf z}|&\leq&|\alpha|\cdot|\lambda-\tau|+|{\bf w}^{\rm H}{\bf z}|\\
&=&\cos\angle({\bf x},{\bf y})\cdot|\lambda-\tau|+\cos\angle({\bf w},{\bf z})\|{\bf w}\|\cdot\sin\angle({\bf x},{\bf y}),
\end{eqnarray*}
and
\begin{eqnarray}\label{eqn2.4}
\frac{|\alpha(\lambda-\tau)+{\bf w}^{\rm H}{\bf z}|}{\|R{\bf z}\|}&\leq&\frac{|\lambda-\tau|\cos\angle({\bf x},{\bf y})}{\sigma_{\min}(R)\cdot\sin\angle({\bf x},{\bf y})}
+\frac{\cos\angle({\bf w},{\bf z})\|{\bf w}\|}{\sigma_{\min}(R)}\nonumber\\
&\leq&\frac{|\lambda-\tau|}{\sigma_{\min}(A-\tau I)}\cdot\cot\angle({\bf x},{\bf y})
+\frac{\cos\angle({\bf w},{\bf z})\|{\bf w}\|}{\sigma_{\min}(A-\tau I)}.
\end{eqnarray}
On the other hand,
\begin{eqnarray*}
|\alpha(\lambda-\tau)+{\bf w}^{\rm H}{\bf z}|&\geq&\left||\alpha|\cdot|\lambda-\tau|-|{\bf w}^{\rm H}{\bf z}|\right|\\
&=&\big|\cos\angle({\bf x},{\bf y})|\lambda-\tau|-\cos\angle({\bf w},{\bf z})\|{\bf w}\|\cdot\sin\angle({\bf x},{\bf y})\big|,
\end{eqnarray*}
and
\begin{eqnarray}\label{eqn2.5}
\frac{|\alpha(\lambda-\tau)+{\bf w}^{\rm H}{\bf z}|}{\|R{\bf z}\|}&\geq&\frac{\big||\lambda-\tau|\cdot\cos\angle({\bf x},{\bf y})-\cos\angle({\bf w},{\bf z})\|{\bf w}\|\cdot\sin\angle({\bf x},{\bf y})\big|}{\sigma_{\max}(A-\tau I)\cdot\sin\angle({\bf x},{\bf y})}\nonumber\\
&\geq&\left|\frac{|\lambda-\tau|}{\sigma_{\max}(A-\tau I)}\cdot\cot\angle({\bf x},{\bf y})
-\frac{\cos\angle({\bf w},{\bf z})\|{\bf w}\|}{\sigma_{\max}(A-\tau I)}\right|.
\end{eqnarray}
In particular, when $A$ is normal, we have
\begin{equation}
\frac{|\lambda-\tau|}{\sigma_{\max}(A-\tau I)}\cdot\cot\angle({\bf x},{\bf y})\leq\frac{|\alpha(\lambda-\tau)|}{\|R{\bf z}\|}
\leq\frac{|\lambda-\tau|}{\sigma_{\min}(A-\tau I)}\cdot\cot\angle({\bf x},{\bf y}).
\end{equation}

Denote
\begin{equation}\label{27}
\eta_1=\frac{\cos\angle({\bf w},{\bf z})\|{\bf w}\|}{\sigma_{\min}(A-\tau I)}\quad{\rm and}\quad\eta_2=\frac{\cos\angle({\bf w},{\bf z})\|{\bf w}\|}{\sigma_{\max}(A-\tau I)},
\end{equation}
and we notice that
\begin{equation}\label{410}
\cos\angle({\bf w},{\bf z})\|{\bf w}\|\leq\|{\bf w}\|\leq\|A\|,
\end{equation}
i.e., both $\eta_1$ and $\eta_2$ are {\it uniformly bounded}.
From \eqref{eqn2.3}--\eqref{27}, we obtain the following theorem on the relationship between $\sin\angle({\bf x},(A-\tau I){\bf y})$ and $\sin\angle({\bf x},{\bf y})$. Note that the results still hold when ${\bf x}$ and ${\bf y}$ are collinear.
\begin{theorem}\label{Thm2.1}
Let $(\lambda,{\bf x})$ be an eigenpair of $A$ and $\tau\neq\lambda$. Then for any nonzero vector ${\bf y}\in\mathbb{C}^n$,
we have
\begin{equation}\label{eqn2.8}
 \Bigg\{\begin{array}{c}
\sin\angle{({\bf x},(A-\tau I){\bf y})}\geq\frac{\sin\angle({\bf x},{\bf y})}{\sqrt{\sin^2\angle({\bf x},{\bf y})+\left(\frac{|\lambda-\tau|}{\sigma_{\min}(A-\tau I)}\cos\angle({\bf x},{\bf y})
+\eta_1\sin\angle({\bf x},{\bf y})\right)^2}}\\
\sin\angle{({\bf x},(A-\tau I){\bf y})}\leq\frac{\sin\angle({\bf x},{\bf y})}{\sqrt{\sin^2\angle({\bf x},{\bf y})+\left(\frac{|\lambda-\tau|}{\sigma_{\max}(A-\tau I)}\cos\angle({\bf x},{\bf y})
-\eta_2\sin\angle({\bf x},{\bf y})\right)^2}}
\end{array},
\end{equation}
where $\eta_1$ and $\eta_2$ are defined in \eqref{27}. Specifically, when $A$ is normal, we have
 \begin{equation}\label{eqn288}
 \Bigg\{\begin{array}{c}
\sin\angle{({\bf x},(A-\tau I){\bf y})}\geq\frac{\sin\angle({\bf x},{\bf y})}{\sqrt{\sin^2\angle({\bf x},{\bf y})+\left(\frac{|\lambda-\tau|}{\sigma_{\min}(A-\tau I)}\cos\angle({\bf x},{\bf y})\right)^2}}\\
\sin\angle{({\bf x},(A-\tau I){\bf y})}\leq\frac{\sin\angle({\bf x},{\bf y})}{\sqrt{\sin^2\angle({\bf x},{\bf y})+\left(\frac{|\lambda-\tau|}{\sigma_{\max}(A-\tau I)}\cos\angle({\bf x},{\bf y})\right)^2}}
\end{array}.
\end{equation}
\end{theorem}

When $\tau=0$ and $A$ is a nonsingular Hermitian matrix, we have the following corollary which is directly from \eqref{eqn288}.
\begin{cor}
Let $(\lambda,{\bf x})$ be an eigenpair of a nonsingular Hermitian matrix $A$. Then for any nonzero vector ${\bf y}$, we have
\begin{equation}\label{eqn2.9}
 \Bigg\{\begin{array}{c}
\sin\angle{({\bf x},A{\bf y})}\geq\frac{\sin\angle({\bf x},{\bf y})}{\sqrt{\sin^2\angle({\bf x},{\bf y})+\left(\frac{\lambda}{\lambda_{\min}(A)}\right)^2\cos^2\angle({\bf x},{\bf y})
}}\\
\sin\angle{({\bf x},A{\bf y})}\leq\frac{\sin\angle({\bf x},{\bf y})}{\sqrt{\sin^2\angle({\bf x},{\bf y})+\left(\frac{\lambda}{\lambda_{\max}(A)}\right)^2\cos^2\angle({\bf x},{\bf y})
}}
\end{array},
\end{equation}
where $\lambda_{\min}(A)$ and $\lambda_{\max}(A)$ are the smallest and largest eigenvalues of $A$ in magnitude.
\end{cor}

\begin{rem}
We stress that the upper and lower bounds given in \eqref{eqn2.9} are sharper than those in \eqref{eqn2.10}. Indeed, it is straightforward to verify that
$$
\frac{\sin\angle({\bf x},{\bf y})}{\sqrt{\sin^2\angle({\bf x},{\bf y})+\left(\frac{\lambda}{\lambda_{\max}(A)}\right)^2\cos^2\angle({\bf x},{\bf y})
}}\leq\left|\frac{\lambda_{\max}}{\lambda}\right|\sin\angle({\bf x},{\bf y}),
$$
and
$$
\frac{\sin\angle({\bf x},{\bf y})}{\sqrt{\sin^2\angle({\bf x},{\bf y})+\left(\frac{\lambda}{\lambda_{\min}(A)}\right)^2\cos^2\angle({\bf x},{\bf y})
}}\geq\left|\frac{\lambda_{\min}}{\lambda}\right|\sin\angle({\bf x},{\bf y}).
$$
\end{rem}

The following result relates $\sin\angle\big({\bf x},(A-\tau I)\mathcal{K}\big)$ to $\sin\angle({\bf x},\mathcal{K})$ for a general matrix.
\begin{cor}\label{Cor2.2}
Let $(\lambda,{\bf x})$ be an eigenpair of $A$ and let $\tau\neq\lambda$. Given a subspace $\mathcal{K}$, let $\widehat{\bf z}=\widehat{X}^{\rm H}\widehat{\bf y}$, where $\widehat{\bf y}\in\mathcal{K}$ such that $\angle({\bf x},\widehat{\bf y})=\angle({\bf x},\mathcal{K})$. Then we have
\begin{equation}\label{eqn3.3}
\sin\angle\big({\bf x},(A-\tau I)\mathcal{K}\big)\leq\frac{\sin\angle({\bf x},\mathcal{K})}{\sqrt{\sin^2\angle({\bf x},\mathcal{K})+\left(\frac{|\lambda-\tau|}{\sigma_{\max}(A-\tau I)}\cos\angle({\bf x},\mathcal{K})
-\widehat{\eta}_2\sin\angle({\bf x},\mathcal{K})\right)^2}},
\end{equation}
where
$$
\widehat{\eta}_2=\frac{\cos\angle({\bf w},\widehat{\bf z})\|{\bf w}\|}{\sigma_{\max}(A-\tau I)},
$$
and $\widehat{X}$ and ${\bf w}$ are defined in \eqref{eqn2.1}.
Specifically, if $A$ is normal, we have
\begin{equation}\label{eqn383}
\sin\angle\big({\bf x},(A-\tau I)\mathcal{K}\big)\leq\frac{\sin\angle({\bf x},\mathcal{K})}{\sqrt{\sin^2\angle({\bf x},\mathcal{K})+\left(\frac{|\lambda-\tau|}{\sigma_{\max}(A-\tau I)}\cos\angle({\bf x},\mathcal{K})\right)^2}},
\end{equation}
\end{cor}
\begin{proof}
We have from \eqref{eqn2.8} and $\angle({\bf x},\widehat{\bf y})=\angle({\bf x},\mathcal{K})$ that
\begin{eqnarray*}
\sin\angle\big({\bf x},(A-\tau I)\mathcal{K}\big)&\leq&\sin\angle{\big({\bf x},(A-\tau I)\widehat{\bf y}\big)}
\\
&\leq&\frac{\sin\angle({\bf x},\widehat{\bf y})}{\sqrt{\sin^2\angle({\bf x},\widehat{\bf y})+\left(\frac{|\lambda-\tau|}{\sigma_{\max}(A-\tau I)}\cos\angle({\bf x},\widehat{\bf y})-\widehat{\eta}_2\sin^2\angle({\bf x},\widehat{\bf y})\right)^2}}\\
&=&\frac{\sin\angle({\bf x},\mathcal{K})}{\sqrt{\sin^2\angle({\bf x},\mathcal{K})+\left(\frac{|\lambda-\tau|}{\sigma_{\max}(A-\tau I)}\cos\angle({\bf x},\mathcal{K})
-\widehat{\eta}_2\sin\angle({\bf x},\mathcal{K})\right)^2}},
\end{eqnarray*}
and we can prove \eqref{eqn383} by using \eqref{eqn288} as $A$ is normal.
\end{proof}

We are now in a position to consider the convergence of harmonic Ritz values.
The proof is along the line of \cite{JS}:
First, we will show that if $\angle{({\bf x},\mathcal{K})}$ is small
then ${1}/({\widetilde{\lambda}-\tau})$ is an eigenvalue of a matrix $D+F$ that is near $D$. Second, we
make use of Elsner's theorem \cite{Els} to show that $D$ must have an eigenvalue that is near that of
$(A-\tau I)^{-1}$. We first have the following theorem.


\begin{theorem}\label{Thm3.1}
Under the above notations, if $|\lambda-\tau|\neq\tan\angle({\bf x},\mathcal{K})\cdot\cos\angle({\bf w},\widehat{\bf z})\|{\bf w}\|$,
then there is a matrix $F$ satisfying
\begin{equation}\label{eqn3.1}
\|F\|\leq\frac{\kappa(A-\tau I)\cdot\sin\angle({\bf x},\mathcal{K})}{\big||\lambda-\tau|\cos\angle({\bf x},\mathcal{K})-\sin\angle({\bf x},\mathcal{K})\cos\angle({\bf w},\widehat{\bf z})\|{\bf w}\|\big|},
\end{equation}
such that $\frac{1}{\lambda-\tau}$ is an eigenvalue of $D+F$, where $\widehat{\bf z}$ is defined in Corollary \ref{Cor2.2}.
Specifically, if $A$ is normal and $\lambda\neq\tau$,
we have
\begin{equation}\label{eqn}
\|F\|\leq\frac{\kappa(A-\tau I)}{|\lambda-\tau|}\tan\angle({\bf x},\mathcal{K}).
\end{equation}
\end{theorem}
\begin{proof}
By \eqref{eqn3.2} and \eqref{eqn3.3},
\begin{eqnarray}\label{ee416}
\|F\|&\leq&\frac{\sin\angle({\bf x},\mathcal{K})\cdot\|(A-\tau I)^{-1}\|}{\sqrt{\sin^2\angle({\bf x},\mathcal{K})+\left(\frac{|\lambda-\tau|}{\sigma_{\max}(A-\tau I)}\cos\angle({\bf x},\mathcal{K})
-\widehat{\eta}_2\sin\angle({\bf x},\mathcal{K})\right)^2-\sin^2\angle({\bf x},\mathcal{K})}}\nonumber\\
&=&\frac{\sin\angle({\bf x},\mathcal{K})\cdot\|(A-\tau I)^{-1}\|}{\left|\frac{|\lambda-\tau|}{\sigma_{\max}(A-\tau I)}\cos\angle({\bf x},\mathcal{K})
-\widehat{\eta}_2\sin\angle({\bf x},\mathcal{K})\right|}\nonumber\\
&=&\frac{\kappa(A-\tau I)\cdot\sin\angle({\bf x},\mathcal{K})}{\big||\lambda-\tau|\cos\angle({\bf x},\mathcal{K})-\sin\angle({\bf x},\mathcal{K})\cos\angle({\bf w},\widehat{\bf z})\|{\bf w}\|\big|}\nonumber,
\end{eqnarray}
and we obtain \eqref{eqn} from \eqref{eqn3.2} and \eqref{eqn383} when $A$ is normal.
\end{proof}
\begin{rem}
By \eqref{410}, $\cos\angle({\bf w},\widehat{\bf z})\|{\bf w}\|$ is uniformly bounded. Therefore, the condition $|\lambda-\tau|\neq\tan\angle({\bf x},\mathcal{K})\cdot\cos\angle({\bf w},\widehat{\bf z})\|{\bf w}\|$ is not stringent as
the uniform separation condition \eqref{14} is satisfied and $\angle({\bf x},\mathcal{K})\rightarrow 0$.
Indeed, as $\angle({\bf x},\mathcal{K})\rightarrow 0$ such that $\cos\angle({\bf w},\widehat{\bf z})\|{\bf w}\|\ll|\lambda-\tau|\cot\angle({\bf x},\mathcal{K})$, we have from \eqref{eqn3.1} that
\begin{equation}\label{3.4}
\|F\|\lesssim\frac{\kappa(A-\tau I)}{|\lambda-\tau|}\tan\angle({\bf x},\mathcal{K})\longrightarrow 0.
\end{equation}
In other words, $\|F\|$ will be {uniformly bounded} and tend to zero as $\angle({\bf x},\mathcal{K})\rightarrow 0$ and the uniform separation condition \eqref{14} is satisfied. Note that the uniform boundness of $\|E\|$ in Theorem \ref{Thm1.4} can not be guaranteed theoretically; refer to \eqref{eqn17}.

\end{rem}

We will present the following corollary on the convergence of harmonic Ritz values. Here we make use of Elsner's theorem \cite{Els} which says that, given matrices $A$ and $\widetilde{A}$ of order $n$, for any eigenvalue $\lambda$ of $A$, there is an eigenvalue $\nu$ of $\widetilde{A}$ satisfying
\begin{equation}\label{eqn3.5}
|\lambda-\nu|\leq(\|A\|+\|\widetilde{A}\|)^{1-\frac{1}{n}}\|A-\widetilde{A}\|^{\frac{1}{n}}.
\end{equation}

\begin{cor}\label{Cor3.1}
Suppose that $\lambda\neq\tau$ and $\widetilde{\lambda}$ is not an infinite eigenvalue of the matrix pencil $(C,B)$.
If $|\lambda-\tau|\neq\tan\angle({\bf x},\mathcal{K})\cdot\cos\angle({\bf w},\widehat{\bf z})\|{\bf w}\|$, then there is an eigenvalue $\frac{1}{\widetilde{\lambda}-\tau}$ of $D$ such that
\begin{equation}\label{eqn35}
\left|\frac{1}{\lambda-\tau}-\frac{1}{\widetilde{\lambda}-\tau}\right|\leq\left(2{\|(A-\tau I)^{-1}\|}+\|F\|\right)^{1-\frac{1}{m}}\|F\|^{\frac{1}{m}},
\end{equation}
where $m$ is the order of $D$.
\end{cor}
\begin{proof}
The result is from applying Elsner's theorem to $D$ and $D+F$, as well as the fact that $\|D\|\leq \|(A-\tau I)^{-1}\|$.
\end{proof}

It was pointed out that the refined harmonic Ritz vector and the harmonic Ritz value share the same sufficient conditions for convergence \cite[pp.1453]{Jia}, i.e., both require that $\angle({\bf x},\mathcal{K})\rightarrow 0$ and $\|B^{-1}\|$ is uniformly bounded (or $B$ is uniformly nonsingular). From Theorem \ref{Thm3.1}, Corollary \ref{Cor3.1}, and Theorem 5.1 of \cite{Jia}, we conclude that the refined harmonic Ritz vector converges as the distance between ${\bf x}$ and $\mathcal{K}$ tends to zero and $\tau$ is satisfied with the {uniform separation condition} \eqref{14}. Hence, Theorem \ref{Thm3.1} and Corollary \ref{Cor3.1} enhance the bounds in \cite{Jia} for refined harmonic Ritz pairs.
Furthermore, our new bounds have some other possible applications to refined harmonic projection methods. For example, they can be applied to the harmonic Lanczos bidiagonalization method for computing a partial SVD of a large matrix \cite{JN}, improving the bounds there for harmonic Ritz pairs approximating the desired singular triplets.

\section{Conclusion}

Given a search subspace and a target point $\tau$, we focus on the convergence of the harmonic Ritz vector and harmonic Ritz value in this paper. The results given in \cite{Chen,Jia} show that the sufficient conditions for the convergence of harmonic Ritz vector include: \\
{\tt (i)} good enough search subspace; \\
{\tt (ii)} $\|B^{-1}\|$ is uniformly bounded, which can not be guaranteed theoretically, even if $\tau$ is not close to the desired eigenvalue $\lambda$; \\
{\tt (iii)} the other uniform separation condition that the desired eigenvalue is uniformly away from the other approximate eigenvalues; 
while the sufficient conditions for the convergence of harmonic Ritz value and refined harmonic Ritz vector include both {\tt (i)} and {\tt (ii)}.

In this paper, we indicate that the sufficient conditions for the convergence of harmonic Ritz vector include {\tt(i)}, {\tt(iii)}, and\\ 
{\tt(iv)} the uniform separation condition \eqref{14}; 
while the sufficient conditions for the convergence of harmonic Ritz value and refined harmonic Ritz vector include both {\tt (i)} and {\tt (iv)}.

Note that the other uniform separation condition {\tt(iii)}, i.e., the {\it sep} function in each of the a priori bounds for Ritz or harmonic Ritz vector, is not guaranteed to be uniformly bounded away from zero, and may indeed be arbitrarily close to and even equal zero for standard and harmonic projection methods; see examples in \cite{Jia,Stewart}. It is this intrinsic deficiency that leads to the introduction of refined (harmonic) Ritz vector. The greatest merit of refined (harmonic) Ritz vector is that its convergence does not require this condition.

\section*{Acknowledgments}
We would like to express our thanks to Prof. Zhong-xiao Jia for pointing out the reference \cite{Voe} and for very insightful comments that greatly improved the representation of this paper.


{\small
\begin{thebibliography}{99}

\bibitem{Bai} {\sc Z. Bai, J. Demmel, J. Dongarra, A. Ruhe and H. van der Vorst eds.}, {\em Templates for the Solution of Algebraic Eigenvalue Problems: A Practical Guide}, SIAM, 2000.

\bibitem{Chen} {\sc G. Chen, Z. Jia}, {\em An analogue of the results of Saad and Stewart for harmonic Ritz vectors}, J. Comput. Appl. Math., 167 (2004), 493--498.

\bibitem{Els} {\sc L. Elsner}, {\em An optimal bound for the spectral variation of two matrices}, Linear Algebra Appl., 71 (1985), pp. 77--80.

\bibitem{Fre} {\sc R. W. Freund}, {\em Quasi-kernel polynomials and their use in non-Hermitian matrix iterations},
J. Comput. Appl. Math., 43 (1992), pp. 135--158.

\bibitem{GV} {\sc G. H. Golub, C. F. Van Loan},
{\em Matrix Computations}, 4th Edition, The Johns Hopkins
University Press, Baltimore and London, 2013.

\bibitem{JS3} {\sc Z. Jia}, {\em The convergence of generalized Lanczos methods for large unsymmetric eigenproblems},
SIAM J. Matrix Anal. Appl., 16 (1995), pp. 843--862.

\bibitem{Jia} {\sc Z. Jia}, {\em The convergence of harmonic Ritz values, harmonic Ritz vectors,
and refined harmonic Ritz vectors}, Math. Comput., 74 (2005), pp. 1441--1456.

\bibitem{JN} {\sc Z. Jia, D. Niu}, {\em A refined harmonic Lanczos bidiagonalization method and an implicitly restarted algorithm for computing the smallest singular triplets of large matrices}, SIAM J. Sci. Comput., 32 (2010), pp. 714--744.

\bibitem{JS} {\sc Z. Jia, G. W. Stewart},
{\em On the Convergence of Ritz Values, Ritz Vectors, and Refined Ritz Vectors}, UMIACS TR-99-07, 1999.

\bibitem{JS2} {\sc Z. Jia, G. W. Stewart}, {\em An analysis of the Rayleigh¨CRitz method for approximating
eigenspaces}, Math. Comput., 70 (2001), pp. 637--647.

\bibitem{Mor1} {\sc R. B. Morgan}, {\em Computing interior eigenvalues of large matrices}, Linear Algebra Appl., 154--156 (1991), pp. 289--309.

\bibitem{Mor2} {\sc R. B. Morgan, M. Zeng}, {\em Harmonic projection methods for large nonsymmetric eigenvalue problems}, Numer. Linear Algebra Appl., 5 (1998), pp. 33--55.

\bibitem{Paige} {\sc C. C. Paige, B. N. Parlett, and H.A. Van der Vorst}, {\em Approximate solutions and eigenvalue bounds from Krylov subspaces}, Numer. Linear Algebra
Appl., 2 (1995), pp. 115--133.



%
%
%
%
%


%
%
%
%



\bibitem{Saad} {\sc Y. Saad}, {\em Numerical Methods for Large Eigenvalue Problems}, 2nd Edition, SIAM, Philadelphia, 2011.

\bibitem{Stewart} {\sc G. W. Stewart}, {\em Matrix Algorithms II: Eigensystems}, SIAM, Philadelphia, 2001.

\bibitem{Stewart1} {\sc G. W. Stewart},
{\em A generalization of Saad's theorem on Rayleigh-Ritz approximations}, Linear Algebra Appl., 327 (2001), pp. 115--119.


\bibitem{Vec} {\sc E. Vecharynski}, {\em A generalization of Saad¡¯s bound on harmonic Ritz
vectors of Hermitian matrices}, axXiv preprint, axXiv:1512.01584v1, 2015.


\bibitem{Voe} {\sc C. V\"{o}emel}, {\em A note on harmonic Ritz values and their reciprocals}, Numer. Linear Algebra Appl., 17 (2010), pp. 97--108.

\bibitem{Voe2} {\sc C. V\"{o}mel, S. Tomov, O. Marques, A. Canning, L. Wang, and J. Dongarra}, {\em State-of-the-art eigensolvers for electronic structure calculations of large scale nano-systmes}, J. Comput. Phy., 227 (2008), pp. 7113--7124.


%
%

\end{thebibliography}
}
\end{document}